\documentclass[10pt,twoside,a4paper,reqno]{amsart}

\usepackage{amsfonts,amsthm,latexsym,mathtools,amsmath,amssymb,enumerate,appendix,color,hyperref,epsf}
\usepackage{bbm}
\usepackage[capitalize]{cleveref}
\usepackage{float}
\usepackage{xcolor, graphicx, tikz, tikz-cd}
\usepackage{graphicx}
\usepackage{subcaption}
\usepackage[utf8]{inputenc}

\newtheorem{theorem}{Theorem}
\numberwithin{theorem}{section}

\newtheorem{lemma}[theorem]{Lemma}

\newtheorem{conjecture}[theorem]{Conjecture}

\theoremstyle{remark}
\newtheorem{remark}[theorem]{Remark}
\newtheorem{example}[theorem]{Example}
\numberwithin{equation}{section}

\makeatletter
\renewcommand*\env@matrix[1][\arraystretch]{%
  \edef\arraystretch{#1}%
  \hskip -\arraycolsep
  \let\@ifnextchar\new@ifnextchar
  \array{*\c@MaxMatrixCols c}}
\makeatother

\newcommand{\R}{\mathbb{R}}

\newcommand{\C}{\mathbb{C}}
\newcommand{\N}{\mathbb{N}}

\newcommand{\T}{\mathbb{T}}
\newcommand{\prob}{\mathbb{P}}

\begin{document}

\title[Perturbations of twisted Toeplitz matrices]{Structured random perturbations of \\ tridiagonal twisted Toeplitz matrices}

\author[D.~Giandinoto]{Dario Giandinoto}
\address{Department of Mathematics, Stockholm University, SE-106 91 Stockholm,
      Sweden}
\email{dario.giandinoto@math.su.se}

\author[B.~Shapiro]{Boris Shapiro}
\address{Department of Mathematics, Stockholm University, SE-106 91 Stockholm,
      Sweden}
\email{shapiro@math.su.se}

\date{\today}

\keywords{twisted Toeplitz matrices, asymptotics of eigenvalues, tridiagonal matrices, non-self-adjoint operators, random perturbations}
\subjclass[2020]{Primary 15B52 Secondary 15B05, 47B28, 47B35, 47B36}

\begin{abstract}
Twisted Toeplitz matrices constitute a generalization of Toeplitz matrices in the sense that the entries on each diagonal no longer need to be constant, but are given by the values of a continuous function on a partition of $[0,1]$. We study the limiting statistical distribution of the eigenvalues of matrices of the form $R_n(a) = T_n(a) + \sigma_n X_n$, where $T_n(a)$ is a sequence of non-Hermitian tridiagonal twisted Toeplitz matrices, $X_n$ is a sequence of tridiagonal random matrices whose entries have mean $0$ and finite variance, and $\sigma_n\to0$. The limiting distribution turns out to be a two-dimensional measure which is in general different from the push-forward of the Lebesgue measure by the symbol. We also explain how the results could extend to banded non-Hermitian twisted Toeplitz matrices.
\end{abstract}

\maketitle

\section{Introduction}
\label{sec:intro}

Let $a:[0,1] \times \T \rightarrow \C$ be a bivariate function (here $\T$ denotes the unit circle in the complex plane) and assume that $a(x,z)$ is integrable with respect to the second variable. Then for $n \in \N $ the $n \times n$ \emph{twisted Toeplitz matrix} $T_n(a)$ associated to $a$ is defined by
\begin{align}
\label{eq:kms}
(T_n(a))_{i,j} = \hat{a}_{i-j} \left( \frac{\min(i,j)}{n} \right) && i,j = 1,2, \dots, n
\end{align}
where $\hat{a}_k(x)$ is the $k$-th Fourier coefficient of $a(x,\cdot)$,
\[ \hat{a}_k(x) = \frac{1}{2 \pi} \int_0^{2 \pi} a(x,e^{i t}) e^{-ikt} dt. \]
For instance, if $a(x,z) = c(x) z^{-1} + b(x) + d(x) z$ for some functions $b,c,d$, we have a tridiagonal twisted Toeplitz matrix:
\[ T_n(a) = \begin{pmatrix}[1.5]
b \big(\frac1n \big) & c\big(\frac1n \big) & & & & \\
d\big(\frac1n \big) & b\big(\frac2n \big) & c\big(\frac2n \big) & & & \\
 & d\big(\frac2n \big) & b\big(\frac3n \big) & c\big(\frac3n \big) & \\
 & & \ddots & \ddots & \ddots &  \\
 & & & d \big(\frac{n-2}{n} \big) & b \big(\frac{n-1}{n} \big) & c \big(\frac{n-1}{n} \big) \\
 & & & & d\big(\frac{n-1}{n} \big) & b(1)
\end{pmatrix}. \]

While the diagonals of a Toeplitz matrix are constant, the diagonals of a twisted Toeplitz matrix are slow-varying, in the sense that they are obtained through sampling functions on a partition of the interval $[0,1]$. We call $a(x,z)$ the \emph{symbol} of $T_n(a)$. It is straightforward to verify that $T_n(a)$ is a Hermitian matrix if and only if its Fourier coefficients satisfy
\[
\hat a_{-k}(x)=\overline{\hat a_k(x)}, \qquad x\in[0,1],\ k\in\mathbb Z.
\]
We observe that this condition is verified if and only if the function $a(x,z)$ is real-valued. Another immediate observation is that if $a(x,z)$ is independent of $x$ (i.e. $a(x,z)=a(z)$ for all $x \in [0,1]$), then $T_n(a)$ is simply the Toeplitz matrix with symbol $a(z)$. \par 

Matrices of the type \eqref{eq:kms} were considered for the first time by Kac, Murdock and Szeg\H{o} in \cite{KMS}, where they studied the LSD in the Hermitian case, under the assumption that $a(x,z)$ is continuous with respect to $x$. In this article they called \eqref{eq:kms} \emph{generalized Toeplitz matrices}. These matrices have been given a variety of names in the following literature, such as \emph{twisted Toeplitz}, \emph{KMS matrices}, \emph{Toeplitz-like},  or \emph{variable coefficient Toeplitz}. In \cite{Til}, Tilli introduced and studied a similar type of matrices, which he called \emph{locally Toeplitz}. Tridiagonal twisted Toeplitz matrices have also been called \emph{sampling Jacobi matrices} in \cite{BlSt1} and \cite{BlSt2}. \par 
The motivation behind studying twisted Toeplitz matrices stems from physics (\cite{BDHMW, Kac}) and differential equations (\cite{BoSh,Til}).
The spectral asymptotics of such matrices have been extensively studied in the case in which they are normal (i.e. they commute with their adjoint); for an illuminating summary of such results, see \cite{BLM}. However, as was underlined by the authors of \cite{BLM}, there are not many results in the non-normal setting, apart from a rough localization of the eigenvalues (see Theorem 2.11 of \cite{BLM}). To this day, the spectral asymptotics of non-normal twisted Toeplitz matrices remain a mystery, even in the tridiagonal case; the eigenvalues seem to distribute along analytic curves, like in the case of banded Toeplitz matrices, but an analogue of the Schmidt-Spitzer theorem \cite{ScSp} has not been found. The LSD of the eigenvalues has been found only in a couple of special cases (see \cite{BlSt1} and \cite{BlSt2}).  \par 
In this paper we will focus on describing the spectral asymptotics for \emph{structured perturbations} of twisted Toeplitz matrices: if $a(x,z)$ is the symbol of a banded twisted Toeplitz matrix $T_n(a)$, we consider the random matrix
\[ R_n(a) = T_n(a) + \sigma_n X_n, \]
where $X_n$ is a random matrix with the same band structure as $T_n(a)$, whose non-zero entries are i.i.d. random variables with mean $0$ and finite variance, while $\sigma_n$ is a sequence of positive real numbers converging to $0$. In this paper we obtain the spectral asymptotics in the case of tridiagonal twisted Toeplitz matrices; later we explain how we expect our results to extend to the banded case, but we do not have a proof of such statements. \par 
Random perturbations of banded Toeplitz matrices have been extensively studied, for Gaussian perturbations in \cite{SjVo} and for more general perturbations in \cite{BPZ}. In this case the distribution of the eigenvalues is shown to converge to the push-forward of the Lebesgue measure on the circle through the symbol. Similar asymptotics were shown to hold for triangular twisted Toeplitz matrices in \cite{BPZ2}, and in a recent article \cite{Noel} it was generalized to twisted Toeplitz matrices with symbol $a(x,z)$ which is smooth in $z$ and piecewise Hölder-continuous in $x$, discarding even the hypothesis of bandedness. \par 
All the results mentioned above treat full-matrix perturbations; in this article we will show how structured perturbations (tridiagonal perturbations in this case) can yield asymptotics different from both the deterministic case and the full-matrix perturbation. We present numerical examples in Figures \ref{fgr:unsym1} and \ref{fgr:unsym2} to show how the eigenvalues of $R_n(a)$ accumulate on sets which are quite different from the range of the symbol $a([0,1] \times \T)$, which according to \cite{Noel} is the limiting set for full-matrix perturbations. The example in Figure \ref{fgr:unsym2} showcases how it can happen that the two sets are completely disjoint. The limiting distributions, and hence the limiting sets of structured perturbations are described in our main results, as follows.

\begin{figure}[t]
\centering
\includegraphics[width=0.49\textwidth]{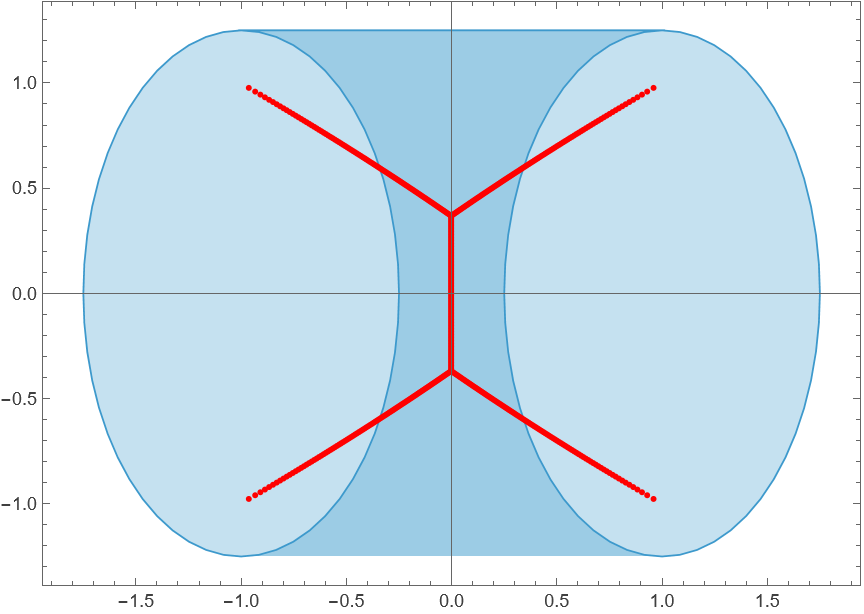}
\includegraphics[width=0.49\textwidth]{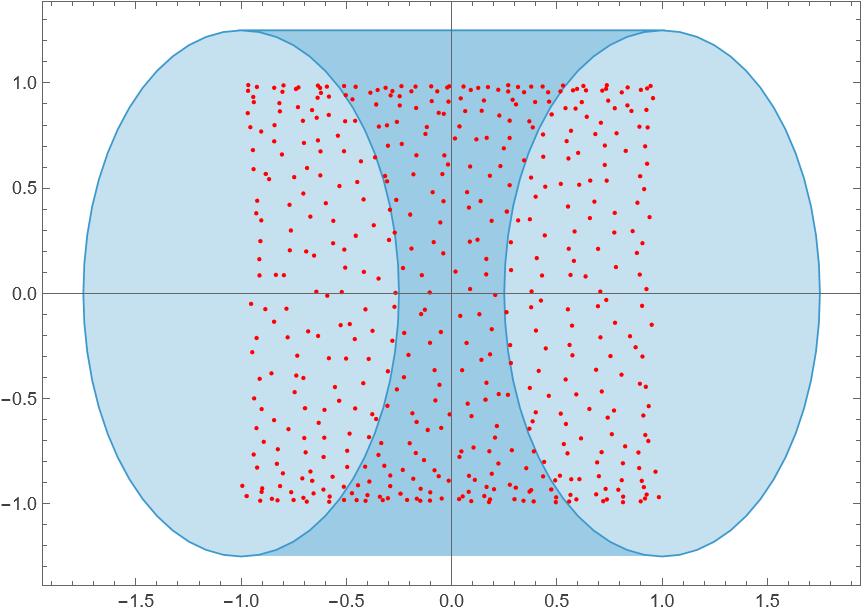}
\caption{In blue, the image of the symbol $a(x,z) = i z^{-1} + 1-2x + \frac{i}{4}z$. On the left, the eigenvalues of $T_{500}(a)$ are pictured in red, while on the right the eigenvalues of $R_{500}(a)$ are plotted. The distribution for the random perturbation is the centered binomial distribution $N(512,0.5)-265$, and $\sigma_n$ was chosen to be $\frac{1}{n}$.}
\label{fgr:unsym1}
\end{figure}

\begin{theorem}
\label{thm:random}
Consider the symbol defined on $[0,1] \times \T$
\[ a(x,z) = d(x) z^{-1} + b(x) + c(x) z, \]
where $b,c,d$ are continuous complex-valued functions on $[0,1]$. Let $\mu_n$ be the sequence of eigenvalue-counting measures
\[
\mu_n = \frac{1}{n} \sum_{\lambda \in \mathrm{sp}(R_n(a))} \delta_{\lambda}.
\]
Then $\mu_n$ converges weakly to the measure $\mu$ defined by
\[
\mu=\int_0^1 \alpha_x~dx,
\]
where $\alpha_x$ denotes the arcsine measure on the complex interval
\[
\bigl[b(x)-2\sqrt{d(x)c(x)},\, b(x)+2\sqrt{d(x)c(x)}\bigr].
\]

\end{theorem}

\begin{figure}[t]
\centering
\includegraphics[width=0.49\textwidth]{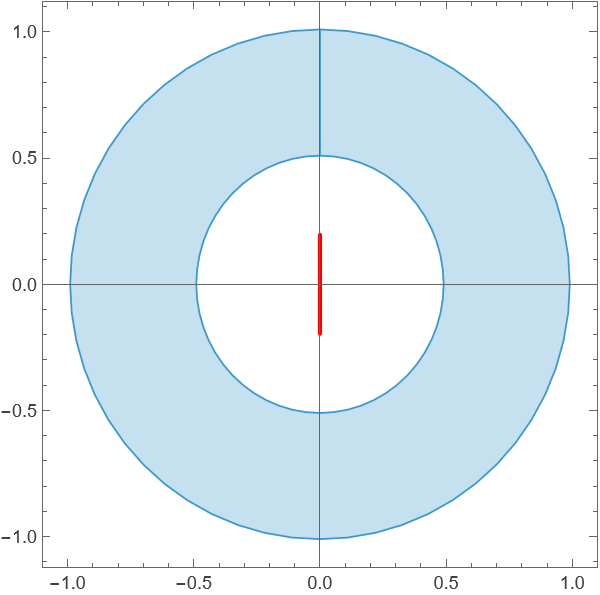}
\includegraphics[width=0.49\textwidth]{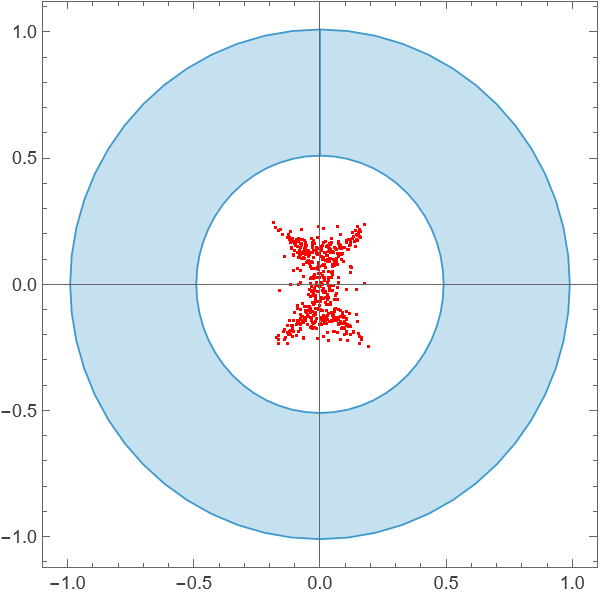}
\caption{In blue, the image of the symbol $a(x,z) = \frac{i}{x+1} z^{-1} + \frac{i}{x^2+100} z$. On the left, the eigenvalues of $T_{500}(a)$ are pictured in red, while on the right the eigenvalues of $R_{500}(a)$ are plotted. The distribution for the random perturbation is the same as in Figure \ref{fgr:unsym1}.}
\label{fgr:unsym2}
\end{figure}

We also obtain a very similar result where, instead of perturbing the entries of the matrix, we randomise the sampling points for our functions. Specifically, we consider $n$ uniform random variables on $[0,1]$
\[ U_1^{(n)}, \cdots, U_n^{(n)} \]
and the corresponding order statistics
\[ 0 \leq x_{1,n} \leq \cdots \leq x_{n,n} \leq 1. \]
We can then define the randomised twisted Toeplitz matrix simply by replacing the deterministic sampling points $i/n$ with the random sampling points $x_{i,n}$: for any symbol $a(x,z)$ on $[0,1]$ we define $\mathcal{T}_n(a)$ to be the matrix with entries

\begin{equation}
\label{eq:rkms}
(\mathcal{T}_n(a))_{i,j} = \hat{a}_{i-j} \bigl( x_{\min(i,j),n} \bigr), \qquad i,j = 1,2, \dots, n.
\end{equation}

We obtain the same asymptotic eigenvalue distribution for tridiagonal randomised twisted Toeplitz matrices as for structured random perturbations.

\begin{theorem}
\label{thm:randomised}
Consider the symbol defined on $[0,1] \times \T$
\[
a(x,z) = d(x) z^{-1} + b(x) + c(x) z,
\]
where $b,c,d$ are continuous complex-valued functions on $[0,1]$. Let $\mu_n$ be the sequence of eigenvalue-counting measures
\[
\mu_n = \frac{1}{n} \sum_{\lambda \in \mathrm{sp}(\mathcal{T}_n(a))} \delta_{\lambda}.
\]
Then $\mu_n$ converges weakly to the measure $\mu$ defined by
\[
\mu=\int_0^1 \alpha_x~dx,
\]
where $\alpha_x$ denotes the arcsine measure on the complex interval
\[
\bigl[b(x)-2\sqrt{d(x)c(x)},\, b(x)+2\sqrt{d(x)c(x)}\bigr].
\]
\end{theorem}

\begin{remark}
The support of the limiting distribution $\mu$ is the set
\[
\Xi(a) = \bigcup_{x \in [0,1]} \bigl[ b(x) - 2 \sqrt{d(x)c(x)},\, b(x) + 2 \sqrt{d(x)c(x)}\bigr].
\]
As we pointed out earlier, this set is in general different from the range of the symbol $a([0,1] \times \T)$ and it can even be disjoint from it. However, in the symmetric case (real or complex), namely when $c(x)=d(x)$ for all $x\in[0,1]$, the two sets $\Xi(a)$ and $a([0,1] \times \T)$ coincide. Indeed, in that case, for $x \in [0,1]$ and $t \in [0, 2 \pi]$,
\[
a(x,e^{it}) = b(x) + c(x)\bigl( e^{it} + e^{-it}\bigr) = b(x) + 2 c(x)\cos t.
\]
Since $\cos t$ ranges from $-1$ to $1$, it follows that $a([0,1] \times \T)=\Xi(a)$.
\end{remark}

The structure of the paper is as follows. In Section \ref{sec:prel} we formulate a number of important technical results. 
The proofs of Theorems \ref{thm:random} and \ref{thm:randomised} are given in Section \ref{sec:thmp}. Finally, in Section \ref{sec:band} we will formulate a conjecture regarding a generalization of our results to banded twisted Toeplitz matrices.

\section{Preliminary results}\label{sec:prel} 

The proofs of Theorems \ref{thm:random} and \ref{thm:randomised} are nearly identical, and they use a "frozen blocks" technique. In the following, we give an outline of the technique for Theorem \ref{thm:random}, but the same argument holds for Theorem \ref{thm:randomised} simply by replacing $R_n(a)$ with $\mathcal{T}_n(a)$. \par 
We consider a sequence $k(n)$ such that $k(n) = o(n)$, $k(n) \rightarrow \infty$, and $k(n)\sigma_n^2\to0$ as $ n \rightarrow \infty$. From now on, we will denote $k(n)=k$ for simplicity.
For any $j \in \N$, $1 \leq j \leq n/k$, we define $R_k^{(j)}$ as the $j$-th $k \times k$ principal block of $R_n(a)$, given by
\[ R_k^{(j)} = ((R_n(a))_{i,j})_{i,j=(j-1)k+1}^{jk}. \]
Then we consider the $k \times k$ Toeplitz matrix $T_k^{(j)}$ with symbol 
\[ a_{j,n}(z) = d(jk/n) z^{-1} + b(jk/n) + c(jk/n) z. \]

These Toeplitz matrices constitute our frozen blocks: our objective is to prove that the spectral distributions of $R_k^{(j)}$ and $T_k^{(j)}$ get arbitrarily close as $n \rightarrow \infty$, with probability $1-o(1)$. Once this is established, the proof of Theorem \ref{thm:random} will consist in a simple inductive argument. \par

Now we aim to formulate more precisely the result which allows us to approximate the eigenvalues of $R_k^{(j)}$ with those of $T_k^{(j)}$. In simple terms, we prove that the spectrum of a tridiagonal Toeplitz matrix is stable under small tridiagonal random perturbations. What we actually prove is an equivalent statement, namely, that the \emph{logarithmic potentials} of the empirical spectral measures associated to the perturbed sequence is arbitrarily close to the ones of the unperturbed one as $n$ approaches infinity, with probability $1-o(1)$. We recall that, given a $n \times n$ matrix $M_n$, the logarithmic potential of its empirical spectral measure can be expressed as
\[ \mathcal{L}_M(z) = \frac{1}{n} \log \vert \det(M_n - zI) \vert. \]
Assume $T_n$ is a sequence of $n \times n$ tridiagonal Toeplitz matrices with symbol $d_0 z^{-1} + b_0 + c_0 z$ and $d_0c_0 \neq 0$. Define the function
\[ \gamma(z) = \log \vert c_0 \vert + \log \max \big( \vert \lambda_1(z) \vert, \vert \lambda_2(z) \vert \big), \]
where $\lambda_1(z)$, $\lambda_2(z)$ are the solutions to the quadratic equation
\[ c_0 \lambda^2 + ( b_0 -z) \lambda + d_0= 0. \]
It is a well established result that, if $T_k$ is a sequence of $k \times k$ Toeplitz matrices with symbol $d_0 z^{-1} + b_0 + c_0 z$, then
\[ \lim_k \frac{1}{k} \log \vert \det (T_k - zI) \vert = \gamma(z). \]
Now we consider a sequence of random tridiagonal matrices
\[ P_n= 
\begin{pmatrix} 
b_1^{(n)} & c_1^{(n)} \\ 
d_1^{(n)} & b_2^{(n)} & c_2^{(n)} \\ 
& \ddots & \ddots & \ddots \\ 
&& d_{n-2}^{(n)} & b_{n-1}^{(n)} & c_{n-1}^{(n)} \\ 
&&& d_{n-1}^{(n)} & b_n^{(n)} \end{pmatrix}. \]

We want to prove the following:

\begin{theorem}
\label{thm:approx}
Define
\[ \gamma(z) = \lim_n \frac{1}{n} \log \vert \det(T_n - zI) \vert. \]
If
\[
\varepsilon_n=\max_{1\le j\le n}\Bigl( |d_j^{(n)}-d_0|+|b_j^{(n)}-b_0|+|c_j^{(n)}-c_0| \Bigr) \longrightarrow 0
\]
in probability, then for every $\epsilon>0$ we have
\begin{equation}
\label{eq:pbound}
\prob \left(
\bigl| \log |\det(P_n-zI)| - n\gamma(z) \bigr| \ge n\epsilon
\right) \longrightarrow 0
\qquad\text{as } n\to\infty.
\end{equation}
\end{theorem}

We will prove the theorem using diagonalization of the transfer matrices associated with the continuant recurrence. Let $D_k^{(n)}(z)$ be the characteristic polynomial of the $k\times k$ leading principal block of $P_n$:
\[
D_k^{(n)}(z):=\det(zI_k-P_k^{(n)}).
\]
Set
\[
D_{-1}^{(n)}(z):=0,\qquad D_0^{(n)}(z):=1.
\]
Then the standard continuant recurrence gives, for $k\ge1$,
\begin{equation}
\label{eq:continuant}
D_k^{(n)}(z)=\bigl(z-b_k^{(n)}\bigr)D_{k-1}^{(n)}(z)-d_{k-1}^{(n)}c_{k-1}^{(n)}D_{k-2}^{(n)}(z),
\end{equation}
where the second term is absent when $k=1$ because $D_{-1}^{(n)}(z)=0$. Since $c_0\neq0$ and $\varepsilon_n\to0$, for all large $n$ we have
\[
|c_j^{(n)}|\ge \frac{|c_0|}{2}>0 \qquad (1\le j\le n-1).
\]
Hence, for $k\ge1$, we may define
\[
Q_k^{(n)}(z):=\frac{D_k^{(n)}(z)}{\prod_{m=1}^{k} c_m^{(n)}},
\]
where, for convenience, we set $c_n^{(n)}:=c_0$ when $k=n$; equivalently, one may restrict the recurrence argument to $k\le n-1$, which is all we need in the sequel. Then \eqref{eq:continuant} becomes
\begin{equation}
\label{eq:Qrec}
Q_k^{(n)}(z)=\frac{z-b_k^{(n)}}{c_k^{(n)}}\,Q_{k-1}^{(n)}(z)-\frac{d_{k-1}^{(n)}}{c_k^{(n)}}\,Q_{k-2}^{(n)}(z), \qquad k\ge1,
\end{equation}
with
\[
Q_{-1}^{(n)}(z)=0,\qquad Q_0^{(n)}(z)=1.
\]

Introduce \[ X_k^{(n)}(z):= \binom{Q_k^{(n)}(z)}{Q_{k-1}^{(n)}(z)}. \]
Then 
\begin{equation}
\label{eq:transfer}
X_k^{(n)}(z)=A_k^{(n)}(z)X_{k-1}^{(n)}(z), \qquad
A_k^{(n)}(z):=
\begin{pmatrix}
\dfrac{z-b_k^{(n)}}{c_k^{(n)}} & -\dfrac{d_{k-1}^{(n)}}{c_k^{(n)}}\\[1.2ex]
1 & 0
\end{pmatrix},
\end{equation}

with 
\[ X_0^{(n)}(z)=\binom{1}{0}. \] 
For the unperturbed Toeplitz matrix, the transfer matrix is 
\[ A(z):= 
\begin{pmatrix} 
\dfrac{z-b_0}{c_0} & -\dfrac{d_0}{c_0}\\[1.2ex] 1 & 0 
\end{pmatrix}. \] 
Its eigenvalues are exactly $\xi_1(z),\xi_2(z)$. \par 
Now, we define the set
\[ \Gamma = \left \{ z \in \C: \vert \xi_1(z) \vert = \vert \xi_2(z) \vert \right \}. \]
The Schmidt-Spitzer theorem \cite{ScSp} tells us this is the limiting set of the eigenvalues of $T_n$. Fix a compact set 
\[ K\Subset \C\setminus\Gamma. \]
Since $|\xi_1(z)|\ne|\xi_2(z)|$ for $z\notin\Gamma$, after relabeling if necessary we may assume
\[ |\xi_1(z)|>|\xi_2(z)| \qquad (z\in K). \] Define \[ \rho(z):=|\xi_1(z)|,\qquad \eta(z):=|\xi_2(z)|. \]
Then the continuous functions $\rho,\eta$ satisfy
\[ m_K:=\inf_{z\in K}\rho(z)>0, \qquad g_K:=\inf_{z\in K}\bigl(\rho(z)-\eta(z)\bigr)>0. \]
Let 
\[ S(z):= \begin{pmatrix} \xi_1(z) & \xi_2(z)\\ 1 & 1 \end{pmatrix}. \] 
Since 
\[ \det S(z)=\xi_1(z)-\xi_2(z)\neq0 \qquad (z\in K), \]
the matrix $S(z)$ is invertible for $z\in K$, and 
\[ S(z)^{-1}A(z)S(z)= \begin{pmatrix} \xi_1(z) & 0\\ 0 & \xi_2(z) \end{pmatrix} =: \Lambda(z). \]
Define the conjugated matrices 
\[ B_k^{(n)}(z):=S(z)^{-1}A_k^{(n)}(z)S(z). \]
Then
\[ B_k^{(n)}(z)=\Lambda(z)+E_k^{(n)}(z), \]
where
\[ E_k^{(n)}(z) = S(z)^{-1} (A_k^{(n)}(z) - A(z)) S(z). \]
Since the entries of $A_k^{(n)}(z) - A(z)$ are dominated by $\varepsilon_n$, there exists a constant $C_K>0$ such that 
\begin{equation}
\label{eq:Ekbound}
\sup_{z\in K}\sup_{1\le k\le n}\max_{i,j}|(E_k^{(n)}(z))_{ij}| \le C_K \varepsilon_n. 
\end{equation}
Let 
\[ Y_k^{(n)}(z):=S(z)^{-1}X_k^{(n)}(z). \]
Then 
\[ Y_k^{(n)}(z)=B_k^{(n)}(z)Y_{k-1}^{(n)}(z), \qquad Y_0^{(n)}(z)=S(z)^{-1}\binom{1}{0} = \frac{1}{\xi_1(z)-\xi_2(z)} \binom{1}{-1}. \] 
For large $n$, by \eqref{eq:Ekbound}, 
\[ \delta_n:=C_K\varepsilon_n<\min\left\{\frac{m_K}{4},\,\frac{g_K}{4}\right\}. \]
Fix $z\in K$ and write
\[ Y_k^{(n)}(z)=\binom{u_k}{v_k}. \]
Consider the cone
\[ \mathcal{C}:=\left\{(u,v)\in\C^2:\ |v|\le |u|\right\}. \]
Since
\[ Y_0^{(n)}(z)=\frac{1}{\xi_1(z)-\xi_2(z)}\binom{1}{-1}, \]
we have $Y_0^{(n)}(z)\in\mathcal{C}$.
\begin{lemma}
\label{lem:cone}
For all sufficiently large $n$, for every $z\in K$ and every $k\ge1$: 
\begin{enumerate} 
\item $Y_k^{(n)}(z)\in\mathcal{C}$; 
\item \[ \bigl(\rho(z)-2\delta_n\bigr)|u_{k-1}| \le |u_k| \le \bigl(\rho(z)+2\delta_n\bigr)|u_{k-1}|. \] 
\end{enumerate}
Consequently, 
\begin{equation}
\label{eq:ugrowth}
\bigl(\rho(z)-2\delta_n\bigr)^n |u_0| \le |u_n| \le \bigl(\rho(z)+2\delta_n\bigr)^n |u_0|.
\end{equation}
\end{lemma} 
\begin{proof} 
Write 
\[ B_k^{(n)}(z)=
\begin{pmatrix} \xi_1(z) & 0\\ 0 & \xi_2(z) 
\end{pmatrix} + 
\begin{pmatrix} e_{11} & e_{12}\\ e_{21} & e_{22} 
\end{pmatrix},
\qquad |e_{ij}|\le\delta_n. \]
Assume $(u_{k-1},v_{k-1})\in\mathcal{C}$, so $|v_{k-1}|\le |u_{k-1}|$. Then
\[ u_k=(\xi_1+e_{11})u_{k-1}+e_{12}v_{k-1}, \]
hence
\[ |u_k| \ge |\xi_1||u_{k-1}|-|e_{11}||u_{k-1}|-|e_{12}||v_{k-1}| \ge (\rho(z)-2\delta_n)|u_{k-1}|. \]
Similarly,
\[ |u_k| \le (\rho(z)+2\delta_n)|u_{k-1}|. \]
Also,
\[ v_k=e_{21}u_{k-1}+(\xi_2+e_{22})v_{k-1}, \]
so
\[ |v_k| \le \delta_n |u_{k-1}|+(\eta(z)+\delta_n)|v_{k-1}| \le (\eta(z)+2\delta_n)|u_{k-1}|. \]
Since
\[ \eta(z)+2\delta_n \le \rho(z)-2\delta_n \]
by the choice of $\delta_n<g_K/4$, we obtain
\[ |v_k|\le |u_k|. \]
Thus $Y_k^{(n)}(z)\in\mathcal{C}$, proving the cone invariance. Iterating the two-sided estimate for $|u_k|$ yields \eqref{eq:ugrowth}. 
\end{proof}
Now 
\[ X_n^{(n)}(z)=S(z)Y_n^{(n)}(z), \qquad Q_n^{(n)}(z)=\text{first coordinate of }X_n^{(n)}(z) =\xi_1(z)u_n+\xi_2(z)v_n. \]
Since $Y_n^{(n)}(z)\in\mathcal{C}$, we have $|v_n|\le |u_n|$, hence
\[ |Q_n^{(n)}(z)| \ge \bigl(\rho(z)-\eta(z)\bigr)|u_n| \ge g_K |u_n|, \]
and
\[ |Q_n^{(n)}(z)| \le \bigl(\rho(z)+\eta(z)\bigr)|u_n| \le M_K |u_n|, \]
where
\[ M_K:=\sup_{z\in K}\bigl(\rho(z)+\eta(z)\bigr)<\infty. \]
Also,
\[ |u_0|=\frac{1}{|\xi_1(z)-\xi_2(z)|} \]
is bounded above and below by positive constants on $K$. Combining these estimates with \eqref{eq:ugrowth}, we obtain constants $c_K,C_K'>0$ such that
\begin{equation}
\label{eq:qnbound}
c_K \bigl(\rho(z)-2\delta_n\bigr)^n \le |Q_n^{(n)}(z)| \le C_K' \bigl(\rho(z)+2\delta_n\bigr)^n, \qquad z\in K.
\end{equation}
Taking logarithms yields
\[ \log c_K + n \log ( \rho(z) - 2\delta_n) \leq \log \vert Q_n^{(n)}(z) \vert \leq n \log ( \rho(z) + 2 \delta_n ) + \log C_K', \]
hence
\[ \frac{1}{n}\log c_K + \log \left( 1 - \frac{2 \delta_n}{\rho(z)} \right) \leq  \frac{1}{n} \log \vert Q_n^{(n)}(z) \vert - \log \rho(z) \leq  \log \left( 1 + \frac{2 \delta_n}{\rho(z)} \right) + \frac{1}{n} \log C_K'. \]
Since $\delta_n \rightarrow 0$ in probability, we obtain that 
\[ \prob \left( \big \vert \log \vert Q_n^{(n)} \vert - n \log \rho(z) \big \vert  \geq n \epsilon \right) = o(1). \]
From the definition of $Q_k^{(n)}$, we have
\[ \frac1n\log|D_n^{(n)}(z)| = \frac1n\sum_{j=1}^n \log|c_j^{(n)}| + \frac1n\log|Q_n^{(n)}(z)|. \]
Because $\sup_j|c_j^{(n)}-c_0|\to0$, 
\[ \sup_{1\le j\le n}\left|\log|c_j^{(n)}|-\log|c_0|\right|\to0, \]
and therefore
\[ \frac1n\sum_{j=1}^n \log|c_j^{(n)}| \longrightarrow \log|c_0|. \]
It is then enough to point out that
\[ \gamma(z) = \log \vert c_0 \vert + \log \rho(z) \]
to conclude the proof of \eqref{eq:pbound}. \par 
Now that we have established that Theorem \ref{thm:approx} holds, we will use it to prove two Lemmas which will be key elements of our proof of Theorems \ref{thm:random} and \ref{thm:randomised}. In order to formulate said Lemmas, we need one more definition. For $x \in [0,1]$, let us consider the sequence $T_n(a_x)$ of matrices with symbol
\[ a_x(z) = d(x) z^{-1} + b(x) + c(x) z. \]
We then define 
\[ \gamma(x,z) = \lim_n \frac{1}{n} \log \vert \det (T_n(a_x) - z I) \vert. \]

We are now ready to prove the first Lemma, which we will use to prove Theorem \ref{thm:random}. We will use the notation established at the beginning of this Section for the "frozen blocks" $T_k^{(j)}$.

\begin{lemma}
\label{lmm:fblock}
Let $R_k^{(j)}$ be the $j$-th principal block of $R_n(a)$. Then, for a.e. $z \in \C$ and any $\varepsilon>0$,
\[
\prob \left( \bigl| \log |\det(R_k^{(j)} - zI)| - k \gamma(jk/n,z) \bigr|  \geq k \varepsilon \right) = o(1).
\]
\begin{proof}
Let us fix $z \in K$ and $j \in \N$, with $1 \leq j \leq n/k$. Denote by $\alpha_r^{(n)}$, $\beta_r^{(n)}$, and $\zeta_r^{(n)}$ the entries of $X_n$ on the lower, main, and upper diagonals, respectively. In light of Theorem \ref{thm:approx}, it is enough to prove that
\begin{multline*}
\varepsilon_k=\max_{1\le r \le k} \bigg(
\left| d\left( \frac{(j-1)k+r}{n} \right) + \sigma_n \alpha_r^{(n)} - d\left( \frac{jk}{n} \right) \right|
+\left| b\left( \frac{(j-1)k+r}{n} \right) + \sigma_n \beta_r^{(n)} - b\left( \frac{jk}{n} \right) \right| \\
+\left| c\left( \frac{(j-1)k+r}{n} \right) + \sigma_n \zeta_r^{(n)} - c\left( \frac{jk}{n} \right) \right|
\bigg) \longrightarrow 0
\end{multline*}
in probability.

By the triangle inequality,
\[
\left| b\left( \frac{(j-1)k+r}{n} \right) + \sigma_n \beta_r^{(n)} - b\left( \frac{jk}{n} \right) \right|
\le
\left| b\left( \frac{(j-1)k+r}{n} \right) - b\left( \frac{jk}{n} \right) \right|
+ \sigma_n |\beta_r^{(n)}|.
\]
Since $b$ is uniformly continuous on $[0,1]$, its modulus of continuity
\[
\omega_b(\delta):=\sup_{|x-y|\le\delta}|b(x)-b(y)|
\]
satisfies $\omega_b(\delta)\to0$ as $\delta\to0$. Hence, uniformly in $j$ and $1\le r\le k$,
\[
\left| b\left( \frac{(j-1)k+r}{n} \right) - b\left( \frac{jk}{n} \right) \right|
\le \omega_b\!\left(\frac{k}{n}\right)\xrightarrow[n\to\infty]{}0.
\]

For the random term, Chebyshev's inequality gives, for every $\epsilon>0$,
\[
\prob\!\left(\sigma_n |\beta_r^{(n)}|\ge \epsilon\right)
\le \frac{\sigma_n^2 \operatorname{Var}(\beta_r^{(n)})}{\epsilon^2}
\le \frac{C\sigma_n^2}{\epsilon^2},
\]
where $C>0$ is independent of $r$ and $n$. Therefore, by the union bound,
\[
\prob\!\left(\max_{1\le r\le k}\sigma_n |\beta_r^{(n)}|\ge \epsilon\right)
\le \sum_{r=1}^k \prob\!\left(\sigma_n |\beta_r^{(n)}|\ge \epsilon\right)
\le \frac{Ck\sigma_n^2}{\epsilon^2}.
\]
Since $k\sigma_n^2\to0$, it follows that
\[
\max_{1\le r\le k}\sigma_n |\beta_r^{(n)}| \longrightarrow 0
\qquad\text{in probability.}
\]

Applying the same argument to the lower and upper diagonals, we conclude that $\varepsilon_k\to0$ in probability. The claim now follows from Theorem \ref{thm:approx}.
\end{proof}
\end{lemma}

We now prove our second Lemma, which we will use to prove Theorem \ref{thm:randomised}.

\begin{lemma}
\label{lmm:rblock}
Let $\mathcal{T}_k^{(j)}$ be the $j$-th principal block of $\mathcal{T}_n(a)$. Then, for a.e. $z \in \C$ and any $\varepsilon>0$,
\[
\prob \left( \bigl| \log |\det(\mathcal{T}_k^{(j)} - zI)| - k \gamma(jk/n,z) \bigr|  \geq k \varepsilon \right) = o(1).
\]

\begin{proof}
As in the previous proof, it is enough to show that
\begin{multline*}
\varepsilon_k=\max_{1\le r \le k} \bigg(
\left| d\left( x_{(j-1)k+r,n} \right) - d\left( \frac{jk}{n} \right) \right|
+\left| b\left( x_{(j-1)k+r,n} \right) - b\left( \frac{jk}{n} \right) \right| \\
+\left| c\left( x_{(j-1)k+r,n} \right) - c\left( \frac{jk}{n} \right) \right|
\bigg) \longrightarrow 0
\end{multline*}
in probability. Since the functions $b$, $c$, and $d$ are uniformly continuous on $[0,1]$, it is enough to prove that
\[
\Delta_k := \max_{1 \leq r \leq k} \left| x_{(j-1)k+r,n}  - \frac{jk}{n} \right| \longrightarrow 0
\]
in probability. In fact, we prove almost sure convergence.

Let
\[
F_n(t):=\frac{1}{n}\sum_{r=1}^n \mathbbm{1}_{\{U_r^{(n)}\le t\}}
\]
be the empirical distribution function. Since the $U_r^{(n)}$ are continuously distributed, ties occur with probability $0$. Hence, almost surely, $x_{(j-1)k+r,n}$ is the $((j-1)k+r)$-th order statistic and
\[
F_n(x_{(j-1)k+r,n})=\frac{(j-1)k+r}{n}.
\]
Therefore,
\[
\left|x_{(j-1)k+r,n}-\frac{jk}{n}\right|
\le \left|x_{(j-1)k+r,n}-F_n(x_{(j-1)k+r,n})\right|+\frac{k - r}{n}
\le \sup_{t\in[0,1]}|F_n(t)-t|+\frac{k}{n}.
\]
Taking the maximum over $r$ yields
\[
\Delta_k \le \sup_{t\in[0,1]}|F_n(t)-t|+\frac{k}{n}.
\]

Now the Dvoretzky--Kiefer--Wolfowitz inequality tells us that for any $\epsilon>0$
\[ \prob \left( \sup_{t \in [0,1]} \vert F_n(t)- F(t) \vert > \epsilon \right) \leq 2 e^{-2n \epsilon^2}, \]
where $F(t)$ is the cumulative distribution function. Since our random variables are uniformly distributed, we have $F(t)=t$. It follows that, for every $\epsilon>0$ and all large $k$ and $n$,
\[
\prob(\Delta_k>\epsilon)
\le \prob\!\left(\sup_{t\in[0,1]}|F_n(t)-t|>\epsilon-\frac{k}{n}\right)
\le 2\exp\!\left(-2n\Bigl(\epsilon-\frac{k}{n}\Bigr)^2\right).
\]

The right-hand side is summable in $n$, so by the Borel--Cantelli lemma we obtain
\[
\Delta_{k(n)}\to0
\qquad\text{almost surely.}
\]
This concludes the proof.
\end{proof}
\end{lemma}

\section{Proof of Theorems \ref{thm:random} and \ref{thm:randomised}}
\label{sec:thmp}
We will now use Lemma \ref{lmm:fblock} to prove Theorem \ref{thm:random}. The proof for Theorem \ref{thm:randomised} using Lemma \ref{lmm:rblock} is exactly the same, so we will not repeat it. First we need a general preliminary result regarding the determinant of $2 \times 2$ block matrices.

\begin{lemma}
\label{lmm:2block}
Consider two matrices $M$ and $N$ of size $r \times r$ and $s \times s$ respectively, and call $M'$ the matrix obtained by removing the last row and column from $M$ and $N'$ the one obtained by removing the first row and column from $N$. Define also two rectangular matrices $X$ and $Y$ of rank 1 and size $s \times r$ and $r \times s$ respectively, whose entries are all $0$ except for $(X)_{1,r}=:x$ and $(Y)_{r,1}=:y$. Then if we define the $2 \times 2$ block matrix
\[ A = \begin{pmatrix}
M & Y \\
X & N
\end{pmatrix} \]
we have
\[ \det A  = \det M \det N - x y \det M' \det N' . \]
\begin{proof}
We prove the formula through the Laplace expansion of the determinant. For a matrix $T$, we denote by $\mu_{i,j}(T)$ the minor obtained by removing the $i$-th row and $j$-th column from $T$. Using Laplace expansion on the $r$-th row of $A$ we get
\[ \det A = \sum_{i=1}^{r+1} (-1)^{i+r} A_{r,i}~\mu_{r,i}(A)  . \]
Now we notice that $\mu_{r,i}(A) = \mu_{r,i}(M) \det N $ for $i \leq r$. Then we have
\[ \det A = \left( \sum_{i=1}^r (-1)^{i+r} M_{r,i}~\mu_{r,i}(M) \right) \det N + (-1)^{2r+1} y~\mu_{r,r+1}(A) = \det M \det N - y~\mu_{r,r+1}(A). \]
Then we observe that if we remove the $r$-th row and $r+1$-th column from $A$ we get the following block matrix:
\[ \begin{pmatrix}
M' & M^r & 0 \\
0 & x & N_1 \\
0 & 0 & N'
\end{pmatrix}, \]
where $M^r$ denotes the $r$-th column of $M$ without $M_{r,r}$ and $N_1$ denotes the first row of $N$ without $N_{1,1}$. It is then obvious to see that
\[ \mu_{r,r+1}(A) = x \det M' \det N', \]
and that concludes the proof.
\end{proof}
\end{lemma}

We are now ready to prove Theorem \ref{thm:random} using Lemmas \ref{lmm:fblock} and \ref{lmm:2block}, and an inductive argument.

\begin{proof}[Proof of Thm. \ref{thm:random}]
Throughout the proof, we denote with $R_k^{(j)'}$ the principal block of $R_k^{(j)}$ obtained by removing its first row and column. First, consider the matrix $R_{2k}(a)$, which can be viewed as a block matrix:
\[
R_{2k}(a)=
\begin{pmatrix}
R_k^{(1)} & C \\
D^T & R_k^{(2)}
\end{pmatrix},
\]
where $C$ and $D$ are $k\times k$ matrices whose entries are all $0$ except for $C_{k,1}=c(k/n)$ and $D_{k,1}=d(k/n)$. We can therefore apply Lemma \ref{lmm:2block}, obtaining
\[
\det (R_{2k}(a)-zI)=\det (R_k^{(1)}-zI)\det (R_k^{(2)}-zI)-c(k/n)d(k/n)\det (R_{k-1}^{(1)}-zI)\det (R_k^{(2)'}-zI).
\]
Taking modules and logarithms, we get
\begin{multline*}
\log |\det (R_{2k}(a)-zI)| = \\
\log |\det (R_k^{(1)}-zI)| + \log |\det (R_k^{(2)}-zI)| \\
+ \log \left|1-\frac{c(k/n)d(k/n)\det (R_{k-1}^{(1)}-zI)\det (R_k^{(2)'}-zI)}
{\det (R_k^{(1)}-zI)\det (R_k^{(2)}-zI)}\right|.
\end{multline*}
Set
\[
\Theta_k^{(1)}(z):=
\frac{c(k/n)d(k/n)\det (R_{k-1}^{(1)}-zI)\det (R_k^{(2)'}-zI)}
{\det (R_k^{(1)}-zI)\det (R_k^{(2)}-zI)}.
\]

In order to continue our proof, we need to show that the quantity $1 - \Theta_k^{(j)}(z)$ is bounded with probability $1-o(1)$. We do that in the following Lemma.

\begin{lemma}
\label{lmm:ratios}
Fix a compact set $K\Subset \C\setminus \Sigma$, where
\[
\Sigma:=\left\{z\in\C:\exists x\in[0,1]\text{ such that }|\xi_1(x,z)|=|\xi_2(x,z)|\right\}
\cup
\left\{z\in\C:\exists x\in[0,1]\text{ such that }\xi_1(x,z)^2=1\right\},
\]
and for each $x\in[0,1]$ the numbers $\xi_1(x,z),\xi_2(x,z)$ are the roots of
\[
\lambda^2-\frac{z-b(x)}{c(x)}\lambda+\frac{d(x)}{c(x)}=0
\]
ordered so that $|\xi_1(x,z)|>|\xi_2(x,z)|$ on $K$.

Then, for every fixed block index $j$ and every $\varepsilon>0$,
\[
\prob\!\left(
\left|
\frac{\det(R_{k-1}^{(j)}-zI)}{\det(R_k^{(j)}-zI)}
-\frac{1}{c(jk/n)\,\xi_1(jk/n,z)}
\right|>\varepsilon
\right)=o(1),
\]
and
\[
\prob\!\left(
\left|
\frac{\det(R_k^{(j)'}-zI)}{\det(R_k^{(j)}-zI)}
-\frac{1}{d(jk/n)\,\xi_1(jk/n,z)}
\right|>\varepsilon
\right)=o(1),
\]
uniformly for $z\in K$.

Consequently,
\[
\Theta_k^{(j)}(z):=
\frac{c(jk/n)d(jk/n)\det(R_{k-1}^{(j)}-zI)\det(R_k^{(j+1)'}-zI)}
{\det(R_k^{(j)}-zI)\det(R_k^{(j+1)}-zI)}
\]
satisfies
\[
\Theta_k^{(j)}(z)\longrightarrow
\frac{1}{\xi_1(jk/n,z)^2}
\qquad\text{in probability, uniformly for }z\in K.
\]
In particular, there exists $\delta_K>0$ such that
\[
\prob\!\left(|1-\Theta_k^{(j)}(z)|\ge \delta_K \right)=1-o(1)
\qquad (z\in K).
\]
\end{lemma}

\begin{proof}
We only prove the first ratio estimate; the second one is identical after reversing the order of the basis vectors in the block, which exchanges the lower and upper diagonals.

Let $x_j:=jk/n$. For the frozen Toeplitz matrix with symbol
\[
a_{x_j}(w)=d(x_j)w^{-1}+b(x_j)+c(x_j)w,
\]
the normalized continuants satisfy
\[
Q_r^{(j)}(z)\sim \kappa_j(z)\,\xi_1(x_j,z)^r
\]
on $K$, by the cone argument used in the proof of Theorem~\ref{thm:approx}. More precisely, the same argument gives
\[
\prob\!\left(
\left|
\frac{Q_{k-1}^{(j)}(z)}{Q_k^{(j)}(z)}-\frac{1}{\xi_1(x_j,z)}
\right|>\varepsilon
\right)=o(1)
\]
uniformly on $K$.

Since
\[
\det(R_k^{(j)}-zI)=\left(\prod_{\ell=1}^k c_\ell^{(j)}\right)Q_k^{(j)}(z),
\qquad
\det(R_{k-1}^{(j)}-zI)=\left(\prod_{\ell=1}^{k-1} c_\ell^{(j)}\right)Q_{k-1}^{(j)}(z),
\]
we obtain
\[
\frac{\det(R_{k-1}^{(j)}-zI)}{\det(R_k^{(j)}-zI)}
=
\frac{1}{c_k^{(j)}}\frac{Q_{k-1}^{(j)}(z)}{Q_k^{(j)}(z)}.
\]
Because the entries in the block are frozen at scale $k/n=o(1)$ and the perturbation amplitude is $\sigma_n\to0$, we have
\[
c_k^{(j)}\longrightarrow c(x_j)
\]
in probability, uniformly in the block. Hence
\[
\frac{\det(R_{k-1}^{(j)}-zI)}{\det(R_k^{(j)}-zI)}
\longrightarrow \frac{1}{c(x_j)\xi_1(x_j,z)}
\]
in probability, uniformly on $K$.

For the trailing minor, we apply the same argument to the reversed block. Reversal exchanges the roles of $c$ and $d$, so
\[
\frac{\det(R_k^{(j)'}-zI)}{\det(R_k^{(j)}-zI)}
\longrightarrow \frac{1}{d(x_j)\xi_1(x_j,z)}
\]
in probability.

Multiplying the two ratio limits yields
\[
\Theta_k^{(j)}(z)\longrightarrow \frac{1}{\xi_1(x_j,z)^2}
\]
in probability.

Finally, since $K\Subset \C\setminus\Sigma$, the continuous function
\[
(z,x)\mapsto 1-\xi_1(x,z)^{-2}
\]
does not vanish on $[0,1]\times K$. Therefore
\[
\delta_K:=\inf_{x\in[0,1],\,z\in K}
\left|1-\xi_1(x,z)^{-2}\right|>0.
\]
Hence
\[
\prob\!\left(|1-\Theta_k^{(j)}(z)|\ge \delta_K/2\right)=1-o(1),
\]
which is the desired non-vanishing estimate.
\end{proof}

Applying Lemma \ref{lmm:ratios} with $j=1$, we obtain
\[
\Theta_k(z)\longrightarrow \frac{1}{\xi_1(k/n,z)^2}
\]
in probability, for $z\in K$. Since $K\Subset \C\setminus\Sigma$, there exists $\delta_K>0$ such that
\[
\prob\bigl(|1-\Theta_k(z)|\ge \delta_K\bigr)=1-o(1).
\]
In particular,
\[
\log |1-\Theta_k(z)|=O(1)
\]
with probability $1-o(1)$, uniformly for $z\in K$. Therefore
\[
\frac{1}{2k}\log |\det (R_{2k}(a)-zI)|
=
\frac{1}{2k}\log |\det (R_k^{(1)}-zI)|
+
\frac{1}{2k}\log |\det (R_k^{(2)}-zI)|
+
O(1/k)
\]
with probability $1-o(1)$.

Inserting the limits of the logarithmic potentials from Lemma \ref{lmm:fblock}, we conclude that
\[
\prob\left(
\left|
\log |\det (R_{2k}(a)-zI)|-k(\gamma_1+\gamma_2)
\right|\ge 2k\varepsilon
\right)=o(1).
\]

Now let $m:=\lfloor n/k\rfloor$, so that $n=mk+r$ with $0\le r<k$. Repeated application of Lemma \ref{lmm:2block} yields
\begin{equation}
\label{eq:globalblock}
\log |\det(R_{mk}(a)-zI)|
=
\sum_{j=1}^m \log |\det(R_k^{(j)}-zI)|
+
\sum_{j=1}^{m-1}\log |1-\Theta_k^{(j)}(z)|,
\end{equation}
where
\[
\Theta_k^{(j)}(z):=
\frac{c(jk/n)d(jk/n)\det(R_{k-1}^{(j)}-zI)\det(R_k^{(j+1)'}-zI)}
{\det(R_k^{(j)}-zI)\det(R_k^{(j+1)}-zI)}.
\]
By Lemma \ref{lmm:ratios}, for every fixed $j$,
\[
\Theta_k^{(j)}(z)\longrightarrow \frac{1}{\xi_1(jk/n,z)^2}
\]
in probability. Since $z\in K\Subset\C\setminus\Sigma$, there exists $\delta_K>0$ independent of $j$ such that
\[
\prob\bigl(|1-\Theta_k^{(j)}(z)|\ge \delta_K\bigr)=1-o(1),
\qquad j=1,\dots,m-1.
\]
Hence each interface term is $O(1)$ with probability $1-o(1)$, and therefore
\[
\sum_{j=1}^{m-1}\log |1-\Theta_k^{(j)}(z)| = O_{\prob}(m).
\]
Since $m\sim n/k$, it follows that
\[
\frac{1}{n}\sum_{j=1}^{m-1}\log |1-\Theta_k^{(j)}(z)| = O_{\prob}(1/k)=o_{\prob}(1).
\]

Moreover, by Lemma \ref{lmm:fblock}, for each fixed $j$,
\[
\frac{1}{k}\log |\det(R_k^{(j)}-zI)|-\gamma(jk/n,z)\longrightarrow 0
\]
in probability. Summing over $j=1,\dots,m$, we obtain
\[
\sum_{j=1}^m \log |\det(R_k^{(j)}-zI)|
=
k\sum_{j=1}^m \gamma(jk/n,z)+o_{\prob}(mk).
\]
Combining this with \eqref{eq:globalblock}, we get
\[
\log |\det(R_{mk}(a)-zI)|
=
k\sum_{j=1}^m \gamma(jk/n,z)+o_{\prob}(n).
\]

Finally, the remaining block has size $r<k$, hence its contribution to the logarithmic potential is negligible after division by $n$. Therefore
\[
\log |\det(R_n(a)-zI)|
=
k\sum_{j=1}^m \gamma(jk/n,z)+o_{\prob}(n).
\]
Since
\[
\frac{k}{n}\sum_{j=1}^m \gamma(jk/n,z)
\]
is a Riemann sum and $mk/n\to1$, we obtain
\[
\frac{1}{n}\log |\det(R_n(a)-zI)| \longrightarrow \int_0^1 \gamma(t,z)\,dt
\]
in probability, for every $z\in\C\setminus\Sigma$. Since $\Sigma$ has planar Lebesgue measure zero, the convergence holds for a.e. $z\in\C$. By the standard criterion based on convergence of logarithmic potentials, this yields weak convergence of the empirical spectral measures.
\end{proof}

\section{Conjecture for banded case}
\label{sec:band}
Until now, we treated exclusively structured perturbations of tridiagonal matrices. A natural extension of our results is to consider banded twisted Toeplitz matrices with banded random perturbations. In this Section we will formulate a natural extension of Theorem \ref{thm:random} to the banded case and present numerical experiments which support our conjecture. \par 
Define the following symbol on $[0,1] \times \T$:
\[ a(x,z) = \sum_{j=-q}^p a_j(x) z^j, \]
where $p,q \geq 1$ and the functions $a_j(x)$ are continuous on $[0,1]$. We define a structured random perturbation of the twisted Toeplitz matrix $T_n(a)$ as the random matrix
\[ R_n(a)= T_n(a) + \sigma_n X_n, \]
where $X_n$ is a random banded matrix, with $(X_n)_{i,j} = 0$ for $i-j < - q$ and $i-j >p$, while for $-q \leq i-j \leq p$ the entries $(X_n)_{i,j}$ are i.i.d. random variables with mean $0$ and finite variance. The sequence $\sigma_n \subset \R_{>0}$ is such that $\sigma_n \rightarrow 0$. The spectral asymptotics of $R_n(a)$ will be formulated in terms of the limiting spectra of the "frozen" Toeplitz matrices. \par 
If we fix a $x \in [0,1]$ the function $a_x := a(x,\cdot)$ is the symbol of a banded Toeplitz matrix; we denote by $T_n(a_x)$ the $n \times n$ Toeplitz matrix associated to the symbol $a_x$. The Schmidt-Spitzer theorem gives us an expression for the limiting set of the eigenvalues of $T_n(a_x)$, as follows: denote with $z_1(\lambda), \cdots, z_{p+q}(\lambda)$ the roots of the polynomial $z^q(a_x(z) - \lambda)$ with respect to $z$, ordered such that
\[
|z_1(\lambda)|\le \cdots \le |z_{p+q}(\lambda)|.
\]
Then the eigenvalues of $T_n(a_x)$ distribute, as $n \rightarrow \infty$, on the set
\[ \Lambda(a_x) = \{ \lambda \in \C: \vert z_q(\lambda) \vert = \vert z_{q+1}(\lambda) \vert \}. \]
Moreover, a theorem of Hirschman \cite{Hir} proves that the sequence of empirical spectral measures of $T_n(a_x)$ is weakly convergent, and even gives an explicit formula for its limiting density, which of course has support on $\Lambda(a_x)$. We then consider
\[ \nu_{n,x} = \frac{1}{n} \sum_{\lambda \in sp(T_n(a_x))} \delta_\lambda, \]
and we denote with $\nu_x$ the weak limit of the sequence $\nu_{n,x}$. We are now ready to formulate our first conjecture, which is a direct generalization of Theorem \ref{thm:random}.

\begin{conjecture}
\label{con:banded}
Consider the symbol on $[0,1] \times \T$
\[ a(x,z) = \sum_{j=-q}^p a_j(x) z^j, \]
where for all $j$ $a_j(x)$ are continuous complex-valued functions on $[0,1]$. Let $\mu_n$ be the sequence of eigenvalue-counting measures
\[ \mu_n = \frac{1}{n} \sum_{\lambda \in \text{sp}(R_n(a))} \delta_{\lambda}. \]
Then $\mu_x$ converges weakly to the measure
\[ \mu = \int_0^1 \nu_x~dx. \]
\end{conjecture}

Just as we did in the tridiagonal case, we can consider a different random set-up, where instead of considering a structured perturbation, we consider random sampling points for the functions in the symbol.

\begin{conjecture}
Consider the symbol on $[0,1] \times \T$
\[ a(x,z) = \sum_{j=-q}^p a_j(x) z^j, \]
where for all $j$ $a_j(x)$ are continuous complex-valued functions on $[0,1]$. Let $\mu_n$ be the sequence of eigenvalue-counting measures
\[ \mu_n = \frac{1}{n} \sum_{\lambda \in \text{sp}(\mathcal{T}_n(a))} \delta_{\lambda}, \]
where $\mathcal{T}_n(a)$ is a randomised twisted Toeplitz matrix as defined in \eqref{eq:rkms}. 
Then $\mu_n$ converges weakly to the measure $\mu$ defined by
\[
\mu=\int_0^1 \nu_x~dx.
\]

\end{conjecture}

\begin{remark}
The support of the limiting measure $\mu$ is the set
\[ \Xi(a) = \bigcup_{x \in [0,1]} \Lambda(a_x), \]
which in general is different from the image of the symbol $a([0,1] \times \T)$. As we mentioned in the Introduction, it has been proven in \cite{Noel} that for a full-matrix random perturbation, the eigenvalues will distribute according to the push-forward of the Lebesgue measure on $[0,1]\times \T$ through the symbol $a$. The support of such a measure is exactly $a([0,1] \times \T)$.
\end{remark}

We will now present some numerical examples which are consistent with Conjecture \ref{con:banded}.

\begin{example}
\label{ex:4diag}
Let us consider the tetradiagonal twisted Toeplitz matrix with symbol
\[ a(x,z) = i \left( 1 - \frac{1}{2}x \right) z^{-1} + 1-x + \frac{i}{4}(1+x)z+ \left( 1+ \frac{1}{2} \right) z^2. \]
In Figure \ref{fgr:4diag}, on the left we plot the eigenvalues of the unperturbed $T_{500}(a)$ in green, the eigenvalues of $R_{500}(a)$ in red, and the range of the symbol in blue. One can see how the set $\Xi(a)$ in this case is not contained in $a([0,1] \times \T)$, but it is contained in the extended range. \par 
On the right we plot the image of the symbol once again, and in red the eigenvalues of the Toeplitz matrices $T_{300}(a_x)$ for $x=0,\frac{1}{4}, \frac{1}{2}, \frac{3}{4}, 1$. This computation shows how the eigenvalues of $R_n(a)$ seem to be distributed according to a measure with support $\Xi(a)$.

\begin{figure}
\centering
\includegraphics[width=0.49\linewidth]{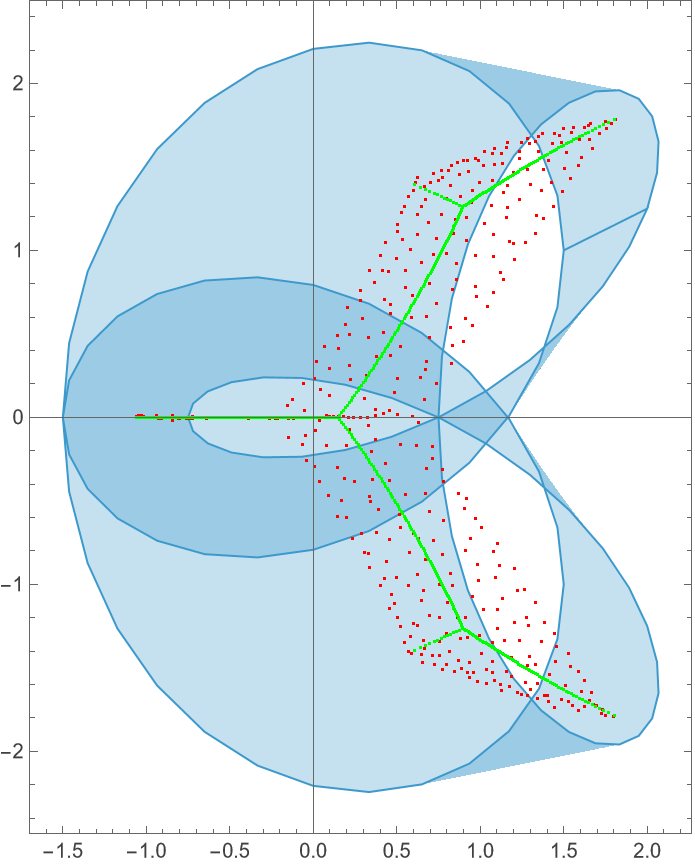}
\includegraphics[width=0.49\linewidth]{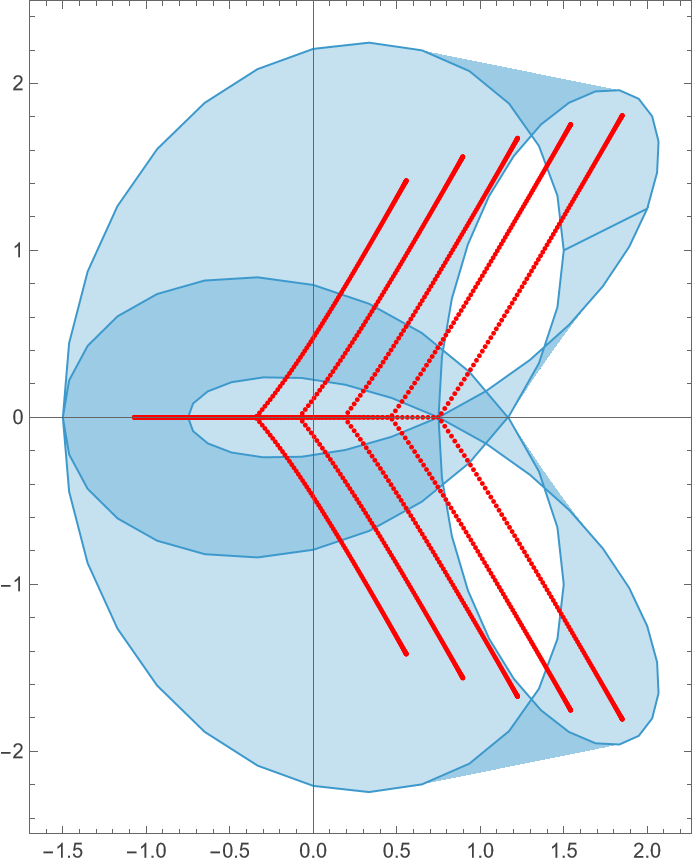}
\caption{Numerical results in Example \ref{ex:4diag}}
\label{fgr:4diag}
\end{figure}
\end{example}

\begin{example}
Our second example involves the pentadiagonal twisted Toeplitz matrix with symbol
\[ a(x,z) = -\frac{1}{5}(x+i) z^{-2} + i(\frac{3}{4}+x) z^{-1} + \frac{i}{3} x +(1-2x) z + \frac{x^2-1}{5} z^2. \]
\label{ex:5diag}
\begin{figure}[b]
\centering
\includegraphics[width=0.49\linewidth]{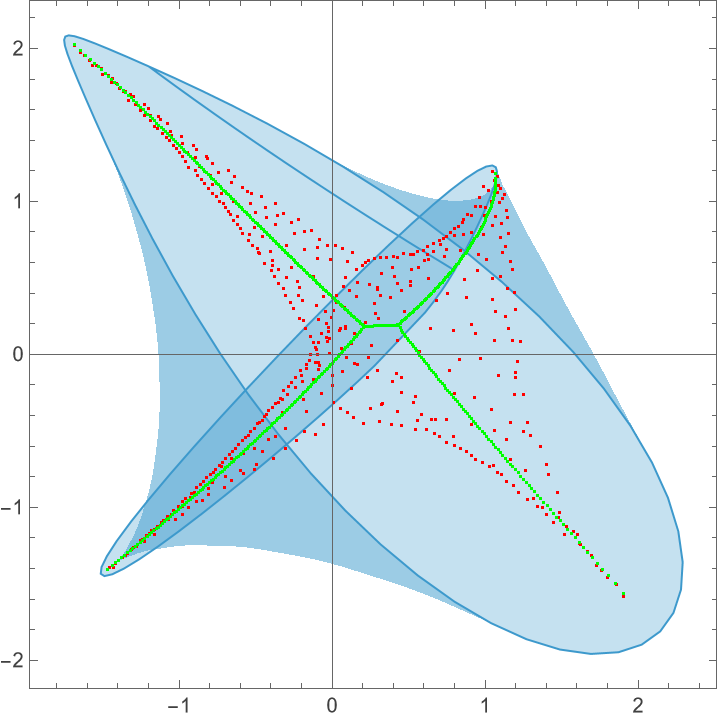}
\includegraphics[width=0.49\linewidth]{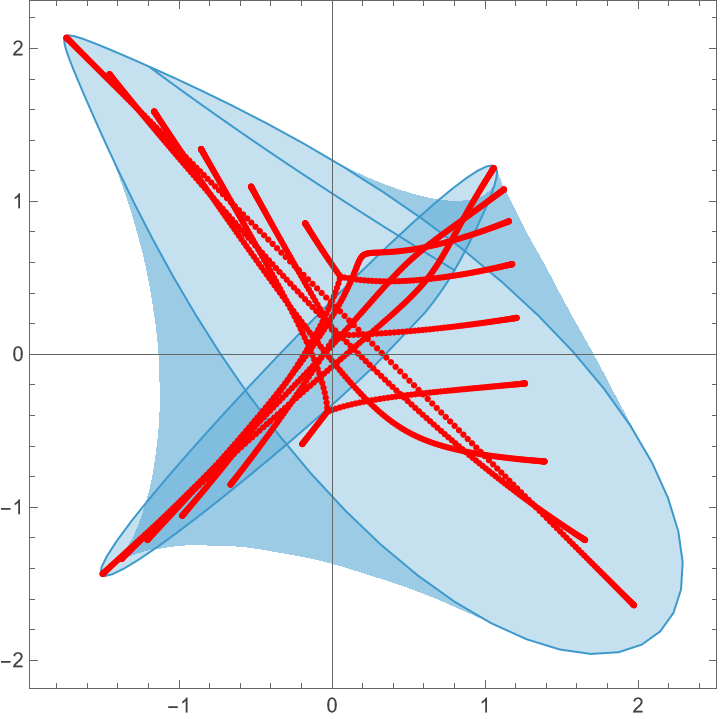}
\caption{ Numerical results in Example \ref{ex:5diag} }
\label{fgr:5diag}
\end{figure}
\end{example}
The color scheme in Figure \ref{fgr:5diag} is the same as in the previous example. One can see the set $\Xi(a)$ seems significantly more complicated in this case. In the image on the right we plot the spectrum of nine "frozen" Toeplitz matrices, revealing how the set in which the eigenvalues distribute still seems to be $\Xi(a)$, even if it is not completely clear due to its more complex structure.

\end{document}